\documentclass[preprint,12pt]{article}
\usepackage{amssymb,amsmath,MnSymbol}
\usepackage{color}
\usepackage{amsthm}
\usepackage{hyperref}
\usepackage{graphicx}
\usepackage{color}

\def\cA{\mathcal A}
\def\cC{\mathcal C}
\def\cN{\mathcal N}
\def\FF{{\mathbb F}}
\def\Alt{\mbox{\rm Alt}}
\def\Sym{\mbox{\rm Sym}}
\def\Out{\mbox{\rm Out}}
\def\PGL{\mbox{\rm PGL}}

\def\Im{\mathrm{Im}}
\def\soc{\mathrm{soc}}
\def\sqr{\scalebox{0.4}{$\square$}}
\newcommand{\AGL}{\mbox{\rm AGL}}
\newcommand{\GL}{\mbox{\rm GL}}
\newcommand{\gr}{\rho}
\newcommand{\gl}{\lambda}
\newcommand{\gs}{\sigma}
\newcommand{\gt}{\tau}
\newcommand{\gk}{\kappa}

\newtheorem{theorem}{Theorem}[section]
\newtheorem{definition}[theorem]{Definition}
\newtheorem{lemma}[theorem]{Lemma}
\newtheorem{proposition}[theorem]{Proposition}
\newtheorem{corollary}[theorem]{Corollary}
\newtheorem{remark}[theorem]{Remark}
\newtheorem{example}[theorem]{Example}
\newtheorem{fact}[theorem]{Fact}

\date{}

\title{\Large{\bf On the primitivity of PRESENT\\ and other lightweight ciphers}}

\author{\small{\textsc{Riccardo Aragona, Marco Calderini}}\\
\small{Dipartimento di Matematica, Universit\`a di Trento}\\
\small{via Sommarive 14, 38123, Povo (TN), Italy}\\
\small{E-mail: ric.aragona@gmail.com, marco.calderini@unitn.it}\\
[10pt]
\small{\textsc{Antonio Tortora} and \textsc{Maria Tota}\footnote{This work was carried out during the last two authors' visit to the University of Trento. They would like to thank the Department of Mathematics for hospitality and GNSAGA (INdAM) for financial support.}}\\
\small{Dipartimento di Matematica, Universit\`a di Salerno}\\
\small{via Giovanni Paolo II, 132 - 84084 - Fisciano (SA), Italy}\\
\small{E-mail: antortora@unisa.it, mtota@unisa.it}}

\begin{document}
\maketitle

\begin{abstract}
We provide two sufficient conditions to guarantee that the round functions of a translation based cipher generate a primitive group. Furthermore, under the same hypotheses, and assuming that a round of the cipher is {\em strongly} proper and consists of $m$-bit S-Boxes, with $m=3,4$ or $5$, we prove that such a group is the alternating group. As an immediate consequence, we deduce that the round functions of some lightweight translation based ciphers, such as the PRESENT cipher, generate the alternating group.
\end{abstract}
\medskip
\small{\textbf{Keywords:}
Lightweight cryptosystems,
Groups generated by round functions,
Primitive groups,
O'Nan-Scott,
Wreath products,
Affine groups}\\
\medskip
\small{\textbf{MSC 2010:} 20B15,
20B35,
94A60
}

\section{Introduction}

Many block ciphers used in real life applications are \emph{translation based ciphers}, which are essentially iterated block ciphers obtained by the composition of several ``round functions'', that is, key-dependent permutations of the message/cipher space. This class of ciphers has been introduced by Caranti, Dalla Volta and Sala, in~\cite{CDVS09}, and contains well-known ciphers like AES~\cite{AES}, SERPENT~\cite{SERPENT} and PRESENT~\cite{PRESENT}, where this latter is one of the most common lightweight ciphers (i.e., ciphers able to run on devices with very low computing power).

Since 1975, when Coppersmith and Grossman~\cite{copp} defined a set of functions which can be adapted for constructing a block cipher, and studied the permutation group generated by them, much attention has been devoted to the group generated by the round functions of a block cipher. In this context, Kaliski, Rivest and Sherman~\cite{KRS88} proved that if such a group is too small, then the cipher is vulnerable to certain cryptanalytic attacks. Later, Paterson~\cite{Pat} showed that the imprimitivity of the group can be used to construct a trapdoor that may be difficult to detect.

For a translation based cipher, in~\cite{CDVS09}, the authors provided two cryptographic conditions on S-Boxes (i.e., the weakly differential uniformity and the strongly anti-invariance) which guarantee the primitivity of the group generated by the round functions of the cipher. Furthermore, in~\cite{ONAN}, using the O'Nan-Scott classification of finite primitive groups, together with another cryptographic assumption, it has been proved that the group in question is the alternating  group. Unfortunately, both these results are not applicable to the PRESENT cipher. Motivated by this, in the present paper we continue the study of the group generated by the round functions of a translation based cipher.

The paper is organised as follows. In Section~\ref{sec:tbciphers}, we recall some definitions and a series of properties and already known results on translation based ciphers. In Section~\ref{sec:primitive}, we deal with primitive groups containing an abelian regular subgroup. In particular, we prove the primitivity of the group $G$ generated by the round functions of a translation based cipher satisfying different cryptographic assumptions with respect to the result given in~\cite{CDVS09}. More precisely, we consider the differential uniformity which allows us to relax the hypothesis on the strongly anti-invariance. Then, in Section~\ref{sec:main}, we provide some additional conditions from which it follows that $G$ is the alternating group. This is, for instance, the case when under the same hypotheses a round of the cipher is {\em strongly} proper and consists of $m$-bit S-Boxes, with $m=3,4$ or $5$. As an immediate consequence, we deduce that the round functions of some lightweight ciphers, such as PRESENT~\cite{PRESENT}, RECTANGLE~\cite{RECTANGLE} and PRINTcipher~\cite{PRINTcipher}, generate the alternating group. Finally, in Section \ref{sec:finrem},
we put in evidence the relationship between the non-linearity of a permutation and the strongly anti-invariance.
\section{Translation based ciphers}
\label{sec:tbciphers}

Let $\mathcal{C}$ be a block cipher acting on a message space $V=(\FF_2)^{d}$, for some $d\geq1$,  and suppose that $V$ coincides with the ciphertext space. Denote by $\mathcal K$ its key space. Then any key $k\in\mathcal K$ induces a permutation $\tau_k$ on $V$ and $\mathcal{C}=\{\tau_k\,|\,k\in\mathcal{K}\}$. We are interested in determining the subgroup $$\Gamma=\Gamma(\mathcal C)=\langle \tau_k\mid k\in\mathcal K\rangle$$
of the symmetric group on $V$ generated by all permutations $\tau_k$.

Unfortunately, the study of $\Gamma$ appears a difficult problem. However, most modern block ciphers are iterated ciphers, i.e., obtained by a composition of several rounds. This allows to investigate an easier permutation group related to $\Gamma$. Assume therefore that $\mathcal C$ is an iterated block cipher. Then each $\tau_k$ is a composition of some permutations of $V$, say $\tau_{k,1},\dots,\tau_{k,l}$. For any round $h$, let
$$\Gamma_h=\Gamma_h(\mathcal C)=\langle\tau_{k,h}\mid k\in\mathcal K\rangle.$$
Thus we can define a new group containing $\Gamma$, that is
$$\Gamma_{\infty}=\Gamma_{\infty}(\mathcal C)=\langle\Gamma_h\mid h\in\{1,\dots,l\}\rangle,$$
which is also known as the {\it group generated by the round functions}.

To better understand the structure of $\Gamma_{\infty}$, we refer to translation based ciphers \cite{CDVS09,ACDVS}. This is a class of iterated block ciphers including some well-known ciphers, as for instance AES \cite{AES} and SERPENT \cite{SERPENT}.

In order to recall the definition of a translation based cipher $\mathcal{C}$ and cite some results on $\Gamma_{\infty}(\mathcal{C})$, we first fix the notation. Let $m,n>1$ and
$$V=V_1\oplus\dots\oplus V_n,$$
where each $V_i$ is isomorphic to $(\FF_2)^m$. As usual, we denote by $\Sym(V)$ and $\Alt(V)$ the symmetric group and the alternating group on $V$, respectively. Given $v\in V$, we write $\sigma_{v} \in \Sym(V)$ for the translation of $V$ mapping $x$ to $x+v$ and denote by $T(V)=\{\sigma_v \mid  v\in  V\}$ the group of all translations of $V$. Clearly $T(V)$ is an elementary abelian regular subgroup of $\Sym(V)$. Also, we denote by $\GL(V)$ the group of linear permutations on $V$ and by $\AGL(V)=\GL(V)\ltimes T(V)$ the group of affine transformations of $V$.

Any $\gamma\in \rm{Sym}(V)$ is called a {\it bricklayer transformation} if, for any $i\in \{1,\ldots,n\}$ and $v=v_1+ \dots + v_n$, with $v_i\in V_i$, there exists a {\it brick} $\gamma_i\in{\rm Sym}(V_i)$ such that $$v\gamma=v_1\gamma_1+\dots+ v_n\gamma_n.$$
In symmetric cryptography, the permutation $\gamma$ is traditionally called parallel S-Box, and each $\gamma_i$ is an $m$-bit S-Box.

A linear map $\lambda:V\rightarrow V$ is a {\it mixing layer} when it is used in composition with a bricklayer transformation. We say that a nontrivial proper subspace $W$ of $V$ is a {\em wall} if it is a sum of some of the subspaces $V_i$. A linear permutation $\lambda\in \GL(V)$ is then a {\it proper mixing layer} if no wall is $\lambda$-invariant. We also say that $\lambda$ is a {\it strongly proper mixing layer} if there are no walls $W$ and $W'$ such that $W\gl=W'$.

\begin{definition}[see \cite{CDVS09}] An iterated block cipher $\mathcal{C}=\{\tau_k \mid k  \in \mathcal{K} \}$ over
  $\FF_2$ is translation based (tb, for short) if the following hold:
       \begin{itemize}
    \item[$(1)$] each $\tau_k$ is the composition of a finite number, say $l$, of round functions $\tau_{k,h}$, with $k\in\mathcal K$ and $h\in\{1,\dots,l\}$, such that each $\tau_{k,h}$ can be written as a composition $\gamma_h\lambda_h\sigma_{\phi(k,h)}$ of three permutations of $V$, where
    \begin{itemize}
    \item[-] $\gamma_h$ is a bricklayer transformation not depending on $k$ and $0\gamma_h=0$,
     \item[-] $\lambda_h$ is a linear permutation not depending on $k$,
    \item[-] $\phi: \mathcal K\times\{1,\dots,l\}\rightarrow V$ is the key scheduling function, so that $\phi(k,h)$ is the $h$-th {\it round key}, given by the master key $k$;
     \end{itemize}

      \item[$(2)$] at least one round is a proper round, that is,
\begin{itemize}
         \item[-] $\lambda_{h}$ is a proper mixing layer for some $h$, and
    \item[-] the map $\phi_{h}: \mathcal K\rightarrow V$ given by $k\rightarrow\phi(k,h)$ is surjective.
     \end{itemize}

    \end{itemize}

\end{definition}

As noted in \cite[Remark 3.3]{CDVS09} the assumption  $0\gamma_h=0$ in (1) is not restrictive. Indeed, we can always include $0\gamma_h$ in the round key addition of the previous round.

In what follows, for a fixed round $h$, we drop the round index $h$ and denote by $\rho\sigma_k$, with $\rho=\gamma\lambda$, the corresponding round function. Furthermore, we refer to a {\em strongly proper round} whenever the mixing layer of a proper round is strongly proper.

\begin{lemma}\label{even}
Let $\mathcal{C}$ be any tb cipher. Then $\Gamma_\infty(\mathcal{C})$ contains only even permutations.
\end{lemma}

\begin{proof}
Let $\tau_{k,h}=\gamma\lambda\sigma_k$ be an arbitrary round function. Clearly $\sigma_k$ is an even permutation and, by \cite[Lemma 2]{We1}, $\lambda$ is also even. We show that $\gamma$ is even, from which the claim follows.

For any $i\in\{1,\ldots, n\}$, let $\overline{\gamma}_i$ be the permutation of $V$ given by
$$(x_1,\ldots,x_n)\overline{\gamma}_i=(x_1,\ldots,x_{i-1},x_i\gamma_i,x_{i+1},\ldots,x_n).$$
Notice that $\gamma=\overline{\gamma}_1\overline{\gamma}_2\ldots\overline{\gamma}_n$. Also, if $\gamma_i$ is the product of $t$ transpositions, then  $\overline{\gamma}_i$ is the product of $t\cdot 2^{m(n-1)}= t\cdot  2^{d-m}$ transpositions. Thus each $\overline{\gamma}_i$ is an even permutation and so is $\gamma$, as required.
\end{proof}

Let $G$ be a transitive permutation group on a set $X$. Recall that a partition $\mathcal{B}$ of $X$ is trivial if $\mathcal{B}=\{X\}$ or $\mathcal{B}=\{\{x\}\,|\,x\in X\}$, and it is $G$-invariant if $Yg\in \mathcal{B}$ for any $Y\in \mathcal{B}$ and $g\in G$. The group $G$ is then called {\em primitive} if it has no nontrivial $G$-invariant partition of $X$. On the other hand, if a nontrivial $G$-invariant partition exists, the group is called imprimitive and the partition is a block system of $G$ on $X$.

We now collect together some results which can be found in \cite{ACDVS} (see Lemma 3.4, Lemma 3.5 and Proposition 3.6).

\begin{lemma}\label{lambda inverse}
Let $\mathcal{C}$ be a tb cipher with a proper round $h$. Then
 \begin{itemize}
 \item[(i)] $\Gamma_h(\mathcal{C})=\langle\rho,T(V)\rangle;$
 \item[(ii)] $\Gamma_h(\mathcal{C})$ is imprimitive on $V$ if and only if there exists a nontrivial proper subspace $U$ of $V$ such that
 $(u+v)\gamma+v\gamma\in U\lambda^{-1}$, for all $u\in U$ and $v\in V$. A block system is then of the form $\{U+v\,|\,v\in V\}$.
\end{itemize}
\end{lemma}

Let $\delta$ and $m$ be positive integers, and let $f:(\FF_2)^m\rightarrow(\FF_2)^m$ be a vectorial Boolean function. Denote by $\hat{f}_u(x)=f(x+u)+f(x)$ the derivative of $f$ in the direction of $u\in(\FF_2)^m$. Recall that $f$ is {\em differentially $\delta$-uniform} (or $\delta$-uniform, for short) if
$$|\{x\in (\FF_2)^m: \hat{f}_u(x)=v\}|\leq\delta,$$
for all $u,v\in (\FF_2)^m$, with $u\neq 0$. By \cite[Fact 3]{CDVS09}, it follows that
$$|{\rm Im}(\hat{f}_u)|\geq\frac{2^{m}}{\delta}.$$

Following \cite{CDVS09}, we say that $f$ is {\it weakly $\delta$-uniform} if
$$|{\rm Im}(\hat{f}_u)|>\frac{2^{m-1}}{\delta},$$
for all $u \in(\mathbb F_2)^m\backslash\{0\}$. Of course, every $\delta$-uniform function is weakly $\delta$-uniform.

Given $1\leq r<m$, we also say that $f$ is {\it $r$-anti-invariant} if $f(0)=0$ and, for any subspace $U$ of $(\mathbb F_2)^m$ such that $f(U)=U$, either $\dim(U)<m-r$ or $U =(\mathbb F_2)^m$. Furthermore, $f$ is said to be {\it strongly $r$-anti-invariant} if, for any two subspaces $U$ and $W$ of $(\mathbb F_2)^m$ such that $f(U)=W$, then either $\dim(U)=\dim(W)<m-r$ or $U =W=(\mathbb F_2)^m$.

The next result is Theorem 4.4 of  \cite{CDVS09}.

\begin{theorem}\label{Th 4.4}
Let $\mathcal{C}$ be a tb cipher with a proper round $h$. Suppose that, for some $1\leq r<m$, each brick of $\gamma$ is
\begin{itemize}
  \item[$(i)$] weakly $2^r$-uniform, and
  \item[$(ii)$] strongly $r$-anti-invariant.
\end{itemize}
Then $\Gamma_h(\mathcal{C})$ is primitive, and hence so is $\Gamma_{\infty}(\mathcal{C})$.
\end{theorem}

\subsection{Some applications to real-life Cryptography}

${\rm AES}$ and ${\rm SERPENT}$ are two translation based ciphers used in real-life applications and their bricks satisfy the hypotheses of Theorem \ref{Th 4.4}; as a consequence, $\Gamma_\infty({\rm AES})$ and $\Gamma_\infty({\rm SERPENT})$ are primitive groups \cite{CDVS09}. Actually, in \cite{We1} and \cite{We2}, with an ad hoc proof, it has been proved respectively that $\Gamma_\infty({\rm AES})=\Alt(V)$ and $\Gamma_\infty({\rm SERPENT})=\Alt(V)$. See also \cite{ONAN, SW} for an AES-like cipher, \cite{GOST} for a GOST-like cipher, and \cite{werdes,kasumi} for ${\rm DES}$ and ${\rm KASUMI}$ respectively.

Other  interesting translation based ciphers are those of type \emph{lightweight}, i.e., ciphers designed to run on devices with very low computing power. The most used in real life applications is PRESENT~\cite{PRESENT} and we would like to apply similar techniques, in order to investigate $\Gamma_\infty({\rm PRESENT})$. In the case of PRESENT, we have $V=(\mathbb F_2)^{64}$. The S-Box used in PRESENT is always the same. It is a 4-bit S-Box $\gamma:(\FF_2)^4 \to (\FF_2)^4$ and its action in hexadecimal notation is given in Table \ref{tab:gamma}. The mixing layer of PRESENT is given in Table \ref{tab:gl} and it is  proper.
\begin{table}[h]
\centering
    \begin{tabular}{|c|c|c|c|c|c|c|c|c|c|c|c|c|c|c|c|c|}
\hline
$x$&0&1&2&3&4&5&6&7&8&9&A&B&C&D&E&F\\
\hline
$x\gamma$&C&5&6&B&9&0&A&D&3&E&F&8&4&7&1&2\\
\hline
\end{tabular}\caption{PRESENT S-Box}\label{tab:gamma}
\end{table}

\begin{table}[h]
\centering
\begin{tabular}{|c|c|c|c|c|c|c|c|c|c|c|c|c|c|c|c|c|}
\hline
$i$&0&1&2&3&4&5&6&7&8&9&10&11&12&13&14&15\\

$i\lambda$&0&16&32&48&1&17&33&49&2&18&34&50&3&19&35&51\\
\hline
\hline
$i$&16&17&18&19&20&21&22&23&24&25&26&27&28&29&30&31\\

$i\lambda$&4&20&36&52&5&21&37&53&6&22&38&54&7&23&39&55\\
\hline
\hline
$i$&32&33&34&35&36&37&38&39&40&41&42&43&44&45&46&47\\

$i\lambda$&8&24&40&56&9&25&41&57&10&26&42&58&11&27&43&59\\
\hline
\hline
$i$&48&49&50&51&52&53&54&55&56&57&58&59&60&61&62&63\\

$i\lambda$&12&28&44&60&13&29&45&61&14&30&46&62&15&31&47&63\\
\hline
\end{tabular}\caption{PRESENT mixing layer}\label{tab:gl}
\end{table}

\begin{remark}\label{present}
{\rm By a computer check (using, for instance, MAGMA \cite{MAGMA}) it is possible to see that the PRESENT mixing layer  is strongly proper and its S-Box is weakly $4$-uniform and strongly $1$-anti-invariant (the strongly anti-invariance is computed with an equivalent permutation that sends $0$ to $0$).}
\end{remark}

By the previous remark, the S-Box of PRESENT does not satisfy the hypotheses of Theorem \ref{Th 4.4}. Nevertheless, in the next section, we will see that $\Gamma_\infty({\rm PRESENT})$ is primitive, as well.

\section{Primitive groups with an abelian regular subgroup}
\label{sec:primitive}

Our first result is similar to Theorem \ref{Th 4.4}. We now consider the differential uniformity instead of the weakly differential uniformity and this leads to relax the assumption on the strongly anti-invariance.

\begin{theorem}
\label{primitivity}
Let $\mathcal{C}$ be a tb cipher with a proper round $h$. Suppose that, for some $1<r<m$, each brick of $\gamma$ is
\begin{itemize}
  \item[$(i)$] $2^r$-uniform, and
  \item[$(ii)$] strongly $(r-1)$-anti-invariant.
\end{itemize}
Then $\Gamma_h(\mathcal{C})$ is primitive, and hence so is $\Gamma_{\infty}(\mathcal{C})$.
\end{theorem}

\begin{proof}
Suppose that $\Gamma_h(\mathcal C)$ is imprimitive. Then, by $(ii)$ of Lemma \ref{lambda inverse}, there exists a nontrivial proper subspace $U$ of $V$ such that $\{U+v\mid v\in V\}$ is a block system of $\Gamma_h(\mathcal C)$ on $V$. Recall that $\gamma\lambda\in \Gamma_h(\mathcal C)$, by $(i)$ of Lemma \ref{lambda inverse}. Thus $U\gamma\lambda=U+v$, for some $v\in V$. Since $0\gamma\lambda=0$ we deduce that $U+v=U$, and therefore $W=U\gamma=U\lambda^{-1}$ is a subspace of $V$.

Let $\pi_i$ be the projection $V\rightarrow V_i$ and denote by $I$ the set of all $i$ such that $\pi_i(U)\neq \{0\}$. Clearly $I\neq\emptyset$. Then either $U\cap V_i=V_i$ for all $i\in I$, or there exists $j\in I$ such that $U\cap V_j\subset V_j$. In the first case, $U=\oplus_{i\in I} V_i$ is a wall. This gives $U\gamma=U$ and consequently $U\lambda=U$, which is impossible because $\lambda$ is a proper mixing layer.

Assume $U\cap V_j\subset V_j$, for some $j\in I$. We claim that $U\cap V_j$ is nontrivial. Notice that $\pi_j(U)\ne \{0\}$, so we can consider $u\in U$ such that $\pi_j(u)=u_j\neq 0$. By Lemma \ref{lambda inverse} $(ii)$, for any $v_j\in V_j$, we have $(u+v_j)\gamma+v_j\gamma\in W$. Since $W$ is a subspace, it follows that $w=u\gamma+(u+v_j)\gamma+v_j\gamma\in W$. Actually $w=u_j\gamma_j+(u_j+v_j)\gamma_j+v_j\gamma_j\in W\cap V_j=(U\cap V_j)\gamma_j$. If $w=0$ for all $v_j$, then the map $\hat{\gamma_j}_{u_j}$ is constant, against the fact that $\gamma_j$ is $2^r$-uniform and $r<m$. Hence $w\neq 0$ for some $v_j$, and $0\neq w\gamma_j^{-1}\in U\cap V_j$.

Now $U\cap V_j$ and $W\cap V_j$ are nontrivial proper subspaces of $V_j$ such that $(U\cap V_j)\gamma_j=W\cap V_j$. Moreover, $\gamma_j$ is strongly $(r-1)$-anti-invariant. Thus
\begin{equation}\label{dim}
\dim(U\cap V_j)=\dim (W\cap V_j)<m-r+1.
\end{equation}
Let $u\in (U\cap V_j)\backslash\{0\}$. Again by $(ii)$ of Lemma \ref{lambda inverse}, we have $(u+v_j)\gamma_j+v_j\gamma_j=(u+v_j)\gamma+v_j\gamma\in W\cap V_j$ for any $v_j\in V_j$. Then ${\rm Im}(\hat{\gamma_j}_u)\subseteq W\cap V_j$, where $|{\rm Im}(\hat{\gamma_j}_{u})|\geq 2^{m-r}$ since $\gamma_j$ is $2^r$-uniform. However $0\notin{\rm Im}(\hat{\gamma_j}_{u})$: otherwise $(u+v_j)\gamma_j=v_j\gamma_j$ and $u=0$, since $\gamma_j$ is a permutation. It follows that $|W\cap V_j|\ge 2^{m-r}+1$, which implies that $\dim( W\cap V_j)\ge m-r+1$, in contradiction to (\ref{dim}). This proves that $\Gamma_h(\mathcal C)$ is primitive.
\end{proof}

\begin{remark}\label{recprin}
{\rm The previous theorem, together with Remark \ref{present}, allows us to conclude that $\Gamma_{\infty}({\rm PRESENT})$ is primitive. Similarly for the lightweight ciphers RECTANGLE~\cite{RECTANGLE} and PRINTcipher~\cite{PRINTcipher}. Indeed, these are tb ciphers with a strongly proper mixing layer. Furthermore,  the RECTANGLE 4-bit S-Box and the PRINTcipher 3-bit S-Box satisfy the hypothesis of Theorem \ref{primitivity}.}
\end{remark}

Notice that the structure of $\Gamma_h(\mathcal{C})$ is also known. In fact, by $(i)$ of Lemma \ref{lambda inverse}, the group $\Gamma_h(\mathcal{C})$ contains $T(V)$ which is an {\em elementary} abelian regular subgroup. We are thus able to apply the characterization of finite primitive groups with an abelian regular subgroup \cite[Theorem 1.1, see also Lemma 3.6 for more details]{Li}. For reader's convenience, we restate it in the particular case where the degree is $2^d$, for some $d\geq 1$.

\begin{theorem}\label{Li}
Let $G$ be a primitive permutation group of degree $2^d$, with $d\geq1$. Then $G$ contains an abelian regular subgroup $T$ if and only if either
\begin{itemize}
  \item[$(1)$] $G\leq \AGL(d,2)$, or
  \item[$(2)$] $G=(S_1\times\ldots\times S_c).O.P\quad and\quad T=T_1\times\ldots\times T_c,$ where $c\geq 1$ divides $d$, each $T_i<S_i$ with $|T_i|=2^{d/c}$, the $S_i$ are all conjugate, $O\leq \Out(S_1)\times\ldots\times \Out(S_c)$, $P$ permutes transitively the $S_i$, and one of the following holds:
  \begin{itemize}
  \item[$(i)$] $S_i\simeq \PGL(d,q)$ and $T_i$ is a cyclic group of order $(q^d-1)/(q-1)$, for $q$ a prime power, or
  \item[$(ii)$] $S_i\simeq \Alt(2^{d/c})$ or $\Sym(2^{d/c})$ and $T_i$ is an abelian group of order $2^{d/c}$.
  \end{itemize}
\end{itemize}
\end{theorem}

In part (2) the notation $G=(S_1\times\ldots\times S_c).O.P$ denotes that $N=S_1\times\ldots\times S_c$ is normal in $G$ and $G/N$ is an extension of the group $O$ by the group $P$. Also, according to the O'Nan-Scott classification of finite primitive groups, the group in (1) is of affine type, while the group in (2) is either almost simple, if $c=1$, or a wreath product (in the product action), if $c>1$.

As an immediate consequence of Theorem \ref{Li}, we have the following refinement when $T$ is elementary abelian.

\begin{corollary}\label{Liref}
Let $G$ be a primitive permutation group of degree $2^d$, with $d\geq 1$. Assume that $G$ contains an elementary abelian regular subgroup $T$. Then one of the following holds:
\begin{itemize}
  \item[$(1)$] $G$ is of affine type, that is, $G\leq \AGL(d,2)$;
  \item[$(2)$] $G\simeq \Alt(2^d)$ or $\Sym(2^d)$;
  \item[$(3)$] $G$ is a wreath product, that is,
  $$G=(S_1\times\ldots\times S_c).O.P\quad and\quad T=T_1\times\ldots\times T_c,$$ where $c>1$ divides $d$, each $T_i$ is an abelian subgroup of $S_i$ of order $2^{d/c}$ with $S_i\simeq \Alt(2^{d/c})$ or $\Sym(2^{d/c})$, the $S_i$ are all conjugate, $O\leq \Out(S_1)\times\ldots\times \Out(S_c)$, and $P$ permutes transitively the $S_i$.
\end{itemize}
In particular, if $d\leq 5$, then $G$ cannot be a wreath product.
\end{corollary}

\begin{proof}
Suppose that $G$ is not affine. Then $G$ is as in part (2) of Theorem \ref{Li}. Let us consider the case when $G$ satisfies $(i)$. Since $T$ is an elementary abelian $2$-group and each $T_i$ is cyclic, it follows that $|T_i|=2$. Hence
$$(q^d-1)/(q-1)=q^{d-1}+q^{d-2}+\ldots+q+1=2,$$
which is impossible. Thus $G$ satisfies $(ii)$ and, of course, we have $G\simeq \Alt(2^d)$ or $\Sym(2^d)$ provided that $c=1$.

Finally, if $d\leq 5$, then the socle $\soc(S_i)$ of each $S_i$ is the Klein group, if $S_i\simeq \Alt(4)$ or $\Sym(4)$, and it is trivial otherwise. On the other hand, by Lemma 3.6 in \cite{Li} (see the proof), the socle of $G$ is a non-abelian group given by $\soc(S_1)\times\ldots\times \soc(S_c)$. Therefore $(3)$ cannot  occur in this case.
\end{proof}

Next we will show that a primitive group on $V=(\mathbb{F}_2)^d$ cannot be of affine type, if $d$ is small and the group is generated by two abelian regular subgroups. First we recall a few preliminary results.

By \cite[Theorem 1]{CDVS06}, for any abelian regular subgroup $T$ of the affine group on $(V,+)$, there is a
structure of an associative, commutative, nilpotent ring $(V,\circ,\cdot)$ on $V$, where
the circle operation is given by
$$x \circ v=x+v+x \cdot v,$$
for all $x,v\in V$. If $T$ is also elementary, then $(V,\circ)$ is a vector space over $\mathbb{F}_2$ such that $T$ is the related translation group. Let $T_+$ and $T_\circ$ be the translation groups with respect to the operations $+$ and $\circ$, respectively, and let us consider the following subspaces
$$U_\circ=\{v\in V\mid x\cdot v=0\text{ for all }x\in V\}$$
and
$$W_\circ=\langle x\cdot v\mid x,v\in V\rangle$$
of $V$. By \cite[Proposition 2.1.6]{Ca15}, if $T_+\neq T_\circ$, then $1\leq \dim(U_\circ)\leq d-2$. Hence, we have $d>2$. Furthermore, slight modifications of Theorem 2.1.18 and Corollary 2.1.29, in \cite{Ca15}, allows us to state that $W_\circ\leq U_\circ$ when $d=3,4$ or $5$; in addition, if $\dim(U_\circ)=d-2$, then $\dim(W_\circ)=1$. It follows that $\dim(W_\circ)=1$ if $d=3$ or $4$, and $\dim(W_\circ)=1$ or $2$ if $d=5$.

\begin{proposition}\label{affine}
Let $d=3,4$ or $5$, and let $G$ be a primitive group on $V=(\mathbb{F}_2)^d$. If $G$ is generated by two elementary abelian regular subgroups, then $G\simeq \Alt(V)$ or $\Sym(V)$.
\end{proposition}

\begin{proof}
Suppose that $G$ is neither isomorphic to $\Sym(V)$ nor to $\Alt(V)$. Then, by Corollary \ref{Liref}, we may assume $G\leq \AGL(V,+)$. Since $G$ is generated by two elementary abelian regular subgroups, by \cite[Theorem 1]{CDVS06}, we may also assume $G=\langle T_\circ,T_{\sqr}\rangle$, where $T_\circ$ and $T_{\sqr}$ are the translations groups with respect to the operations $\circ$ and \scalebox{0.6}{$\square$}. Suppose  $T_\circ=T_+$. Since $\dim(U_{\sqr})\geq 1$, there exists $x\in V\backslash\{0\}$ such that $x$ \scalebox{0.6}{$\square$} $v=x+v$, for all $v\in V$. Thus $\{\{0,x\}+v\mid v\in V\}$ is a block system for $G$ and $G$ is imprimitive, a contradiction. Similarly, if $T_{\sqr}=T_+$. Hence $T_\circ$ and $T_{\sqr}$ are both different from $T_+$, and therefore we can apply the results mentioned above. Keeping the same notation, we have $\dim(W_\circ +W_{\sqr})<d$, so that $W_\circ +W_{\sqr}$ is a proper subspace of $V$.

For any $u\in V$, let $\gt_u^\circ$ and $\gt_u^{\sqr}$ be the translations by $u$ with respect to $\circ$ and \scalebox{0.6}{$\square$}, respectively. Then $\gt_u^\circ=\kappa_u^\circ\sigma_u$, for some $\kappa_u^\circ\in \GL(V,+)$ and $\sigma_u\in T_+$. Take any $x\in W_\circ +W_{\sqr}$. Recall that $x\circ u=x+u+x\cdot u$. Thus
$$x\cdot u=x\gt_u^\circ+x+u=x\kappa_u^\circ+x$$
and, since $x\cdot u\in W_{\circ}$, we have $x\kappa_u^\circ=x\cdot u+ x\in W_\circ +W_{\sqr}$. It follows that
$$((W_\circ +W_{\sqr})+v)\gt^\circ_u=(W_\circ +W_{\sqr})\gk_u^\circ+v\gk_u^\circ+u=(W_\circ +W_{\sqr})+v\gk_u^\circ+u,$$
for all $v\in V$. Clearly, the same holds for $\gt^{\sqr}_u$. This proves that $\{(W_\circ +W_{\sqr})+v\,|\,v\in V\}$ is a block system for $G$, our final contradiction.
\end{proof}

We point out that the previous result cannot be extended to $d=6$, indeed a counterexample can be constructed with the aid of MAGMA \cite{MAGMA}.

\section{The main results}
\label{sec:main}

In this section, for a tb cipher $\cC$ over $(\mathbb{F}_2)^{mn}$ with a strongly proper round $h$, we wonder when $\Gamma_{\infty}(\cC)$ is the alternating group. Of course, it is enough to consider $\Gamma_h(\cC)$. Thus, assuming that $\Gamma_h(\cC)$ is primitive, by Corollary \ref{Liref} and Lemma \ref{even} we can restrict our attention to the cases where $\Gamma_h(\cC)$ is of affine type or a wreath product.

Following \cite{Ca15}, we say that a vectorial Boolean function $f$\,:\,$V\rightarrow V$ is {\em anti-crooked} (AC, for short) if, for any $a\in V\backslash\{0\}$, the set $$\mathrm{Im}(\hat{f}_a) = \{f(x+a)+ f(x)\mid x\in V \}$$ is not an affine subspace of $V$.

In \cite{ACDVS}, part (2) of Theorem 4.5, the AC condition has been used to avoid that $\Gamma_{\infty}(\mathcal{C})$ is of affine type. This result remains valid for $\Gamma_{h}(\mathcal{C})$.

\begin{proposition}\label{Th 4.5}
Let $\mathcal{C}$ be a tb cipher with a proper round $h$ where any brick is AC. If $\Gamma_{h}(\mathcal{C})$ is primitive, then it is not of affine type.
\end{proposition}

However the bricks of some tb ciphers, as for example the S-Box of PRESENT, are not AC. We therefore provide the following alternative condition.

\begin{proposition}\label{gruppetto}
Let $\cC$ be a tb cipher over $V=(\mathbb{F}_2)^{mn}$, with $m\geq 3$ and $n\geq 2$. Suppose that there exists a brick $\gamma_i$ corresponding to a proper round $h$ such that
\begin{equation}\label{W's cond}
\Alt(V_i)\subseteq \langle T(V_i),\gamma_iT(V_i)\gamma_i^{-1}\rangle.
\end{equation}
If $\Gamma_h(\cC)$ is a primitive group, then it is not of affine type.
\end{proposition}

\begin{proof}
Suppose to the contrary that $\Gamma_h(\cC)\leq\AGL(mn,2)$ (see Corollary \ref{Liref}). Let $\tau_i\in \Alt(V_i)$ be a 3-cycle and, for any  $x=(x_1,\ldots,x_{n})\in V$, define $\tau\in Sym(V)$ to be the permutation
$$x \mapsto (x_1,\ldots,x_{i-1},x_i\tau_i,x_{i+1},\ldots, x_{n}).$$
The claim will follow once it is shown that $\tau\in \Gamma_h(\cC)$. In fact, the minimal degree of $\AGL(mn,2)$ is $2^{mn-1}$, where the minimal degree is the minimum number of elements moved by a non-identity permutation (see, for instance, \cite[Section 3.3]{DixMor}). On the other hand, $\tau$ moves exactly
$3\cdot 2^{m(n-1)}$ elements. This is impossible because $3\cdot 2^{m(n-1)}<2^{mn-1}$, for any $m\geq 3$ and $n\geq 2$.

By (\ref{W's cond}), we have $\tau_i=\tau_{i_1}\tau_{i_2}\ldots \tau_{i_s}$ where each $\tau_{i_r}=\sigma_{k_i}$ or $\gamma_i \sigma_{k'_i}\gamma_i^{-1}$, for some $k_i,k'_i\in V_i$. Since $h$ is a proper round, we can consider round keys $k,k'\in V$ such that $k=(0,\ldots,0,k_i,0\ldots,0)$ and $k'=(0,\ldots,0,k'_i,0\ldots,0)$. Put $\rho_k=\rho\sigma_k$, and similarly for $k'$. Then, for any $x\in V$, we have
$$
x\gr_{k'}^{-1}\gr_{k-k'}=x( \gs_{k'}\gl^{-1}\gamma^{-1}\gamma \gl \gs_{k-k'})=x \gs_{k}
$$and
$$
x\gr_{k'}\gr_{k\lambda-k'}^{-1}=x(\gamma \gl \gs_{k'} \gs_{k\lambda-k'}\gl^{-1}\gamma^{-1})=(x\gamma\gl+k\lambda)\gl^{-1}\gamma^{-1}=x\gamma\gs_{k}\gamma^{-1}.
$$
It follows that the permutation
$$x \mapsto (x_1,\ldots,x_{i-1},x_i\tau_{i_r},x_{i+1},\ldots, x_{n})$$
belongs to $\Gamma_h(\cC)=\langle\rho, T(V)\rangle$. We conclude therefore
that $\tau\in\Gamma_h(\cC)$, as desired.
\end{proof}

\begin{remark}
{\rm By a computer check on $4$-bit S-Boxes, one can see that there exist maps which are AC but do not satisfy condition (\ref{W's cond}) and, conversely, maps satisfying condition (\ref{W's cond}) which are not AC.}
\end{remark}

In \cite[Section 7]{ACDVS}, with $\cC$ as in Theorem \ref{Th 4.4}, it has been shown that $\Gamma_h(\mathcal{C})$ cannot be a wreath product (provided that the proper round is strongly proper). In the next result we recall the proof and extend it to tb ciphers satisfying Theorem \ref{primitivity}.

\begin{proposition}\label{altnotpres}
Let $\mathcal{C}$ be a tb cipher with a strongly proper round $h$. If the hypotheses of Theorem $\ref{primitivity}$ (or Theorem $\ref{Th 4.4}$, respectively) are satisfied, then the primitive group $\Gamma_h(\mathcal{C})$ is not a wreath product.
\end{proposition}

\begin{proof}
Of course the group $\Gamma_{h}(\mathcal{C})=\langle \rho,T(V)\rangle$ is primitive, by Theorem \ref{primitivity} (or Theorem \ref{Th 4.4}, respectively). Recall that $\rho=\gamma\lambda$, where $\lambda$ is a strongly proper mixing layer. Suppose that $G=\Gamma_h(\cC)=(S_1\times\ldots\times S_c).O.P$ is the group given in (3) of Corollary \ref{Liref}. Let $S=S_1\times\ldots\times S_c$. Then $T\leq S$ and $G/S$ is a cyclic group generated by $\rho S$. Thus, arguing as in Section 7 of \cite{ACDVS} (see also \cite[Section 6]{ONAN}), we have
$$V=W_1\oplus\ldots\oplus W_c$$
where each $W_i=0T_i\subseteq 0S_i$ is a subspace of $V$ such that $W_i\rho=W_{i+1}$, for $i=1,\ldots, c-1$, and $W_{c}\rho=W_1$.
Furthermore, for any $k\in\{1,\ldots,c\}$, we obtain
\begin{equation}\label{lambda1}
\hat{\gamma}_u(v)=(u+v)\gamma+v\gamma\in W_{k+1}\lambda^{-1}
\end{equation}
for all $u\in W_k$ and $v\in V$.

Let $\pi_i$ be the projection of $V$ onto $V_i$ and take $I_k$ to be the set of all $i$ such that $\pi_i(W_k)\neq \{0\}$. Clearly $I_k\neq\emptyset$. Assume first $W_k\cap V_i=V_i$ for all $k\in\{1,\ldots,c\}$ and $i\in I_k$. Thus $W_k=\oplus_{i\in I_k} V_i$ is a wall, for any $k$. In particular, $W_k\gamma=W_k$. Since $W_k\rho=W_{k+1}$, it follows that $W_k\lambda=W_{k+1}$ which is impossible, being $\lambda$ strongly proper. We may therefore assume $W_k\cap V_j\subset V_j$, for some $k\in\{1,\ldots, c\}$ and $j\in I_k$. Now, if Theorem $\ref{primitivity}$ holds, then we get a contradiction arguing as in Theorem \ref{primitivity} with $U=W_k, W=W_k\gamma=W_{k+1}\lambda^{-1}$, and using (\ref{lambda1}) instead of Lemma \ref{lambda inverse}. Similarly, by \cite[Section 7, part (II)]{ACDVS}, we have a contradiction when Theorem $\ref{Th 4.4}$ is satisfied.
\end{proof}

Notice that in Proposition \ref{altnotpres} it is essential that the proper round is strongly proper.

\begin{example}\label{example}
{\rm
Let $V=V_1\oplus V_2\oplus V_3\oplus V_4$, where each $V_i=(\FF_2)^4$. Consider a tb cipher $\cC$ over $V$ with the inversion map in $\FF_{2^4}$ as S-Box for any round $h$, i.e. $\gamma_i:x\mapsto x^{2^4-2}$. Suppose that there is a unique mixing layer given by the following matrix
$$
\gl=\left[\begin{array}{ccccc}
0&I&0&0\\
0&0&I&0\\
0&0&0&I\\
I&0&0&0\\
\end{array}\right]
$$
where 0 and $I$ are respectively the zero and identity matrices acting on $(\FF_2)^4$. It is known that $\gamma_i$ is weakly $2$-uniform and strongly $1$-anti-invariant, namely $\gamma_i$ satisfies the hypotheses of Theorem \ref{Th 4.4}. It is also easy to verify that $\gl$ is proper. On the other hand, $\lambda$ sends $V_i$ to $V_{i+1}$, for $i=1,2,3,$ and $V_4$ to $V_1$. Thus $\lambda$ is not strongly proper. Furthermore, a computer check shows that $\Gamma_h(\mathcal{C})=\Gamma_{\infty}(\mathcal{C})$ is a wreath product.
}
\end{example}

We can now sum up our conclusions about the simplicity of $\Gamma_{\infty}(\mathcal{C})$, as follows.

\begin{theorem}[see also \cite{ACDVS}]\label{main1}
Let $\mathcal{C}$ be a tb cipher over $V=(\mathbb{F}_2)^{mn}$, with a strongly proper round $h$ such that
the corresponding bricks are AC and satisfy the hypotheses of Theorem $\ref{primitivity}$ (or Theorem $\ref{Th 4.4}$, respectively).
Then $\Gamma_{\infty}(\mathcal{C})=\Alt(V )$.
\end{theorem}

\begin{proof}
By Theorem $\ref{primitivity}$ (or Theorem \ref{Th 4.4}, respectively), the group $\Gamma_{h}(\mathcal{C})$ is primitive. Thus the claim follows applying first Corollary \ref{Liref}, and then Propositions \ref{Th 4.5} and \ref{altnotpres}.
\end{proof}

\begin{theorem}\label{main2}
Let $\mathcal{C}$ be a tb cipher over $V=(\mathbb{F}_2)^{mn}$, with $m\geq 3,n\geq 2$ and a strongly proper round $h$ such that
the corresponding bricks satisfy the hypotheses of Theorem $\ref{primitivity}$ (or Theorem $\ref{Th 4.4}$, respectively). Suppose further that one of these bricks satisfies condition $(\ref{W's cond})$ of Proposition $\ref{gruppetto}$. Then $\Gamma_{\infty}(\mathcal{C})=\Alt(V )$.
\end{theorem}

\begin{proof}
As in Theorem $\ref{main1}$, applying Proposition \ref{gruppetto} instead of Proposition~\ref{Th 4.5}.
\end{proof}

\begin{corollary}\label{cormain2}
Let $\mathcal{C}$ be a tb cipher over $V=(\mathbb{F}_2)^{mn}$ with a strongly proper round $h$ such that
the corresponding bricks satisfy the hypotheses of Theorem $\ref{primitivity}$ (or Theorem $\ref{Th 4.4}$, respectively). Suppose $m=3,4$ or $5$, and $n\geq 2$. Then $\Gamma_{\infty}(\mathcal{C})=\Alt(V )$.
\end{corollary}

\begin{proof}
Let $\gamma_i$ be any brick in the round $h$. Put $V_i=(\mathbb{F}_2)^{m}$ and
$$G=\langle \sigma_k, \gamma_i\sigma_k\gamma_i^{-1} \mid k\in V_i\rangle.$$
According to Theorem \ref{main2}, it is enough to prove that $\Alt(V_i)\subseteq G$. Actually, by Corollary \ref{Liref} and Proposition \ref{affine}, it suffices to show that $G$ is primitive. In fact, $G$ is a subgroup of $\Sym(V_i)$ generated by two elementary abelian regular subgroups, namely $T(V_i)$ and $\gamma_i T(V_i)\gamma_i^{-1}$.

Suppose to the contrary that $G$ is imprimitive. Let $U$ be a nontrivial proper subspace of $V_i$ such that $\{U+v\mid v\in V_i\}$ is a block system of $G$ on $V_i$. Then $(U+v)\gamma_i\sigma_k\gamma_i^{-1}=U+w$, with $v,w\in V_i$. It follows that $v\gamma_i\sigma_k\gamma_i^{-1}\in U+w$ and so $(U+v)\gamma_i\sigma_k\gamma_i^{-1}=U+v\gamma_i\sigma_k\gamma_i^{-1}=U+(v\gamma_i+k)\gamma_i^{-1}$. Hence,
$$(U+v)\gamma_i\sigma_k\gamma_i^{-1}+(v\gamma_i+k)\gamma_i^{-1}=U.$$
Now, if $v=0$ and $k\in U\gamma_i$, then $U\gamma_i\sigma_k\gamma_i^{-1}=U$ because $0\gamma_i=0$. Hence $U\gamma_i+k=U\gamma_i$, which implies that $U\gamma_i$ is a subspace of $V_i$. On the other hand, if $k=v\gamma_i$, we have
$((U+v)\gamma_i+v\gamma_i)\gamma_i^{-1}=U$.
Thus $\Im(\hat{\gamma_i}_{u})\cup\{0\}\subseteq U\gamma_i$, for any $u\in U\backslash\{0\}$.
However $0\notin{\rm Im}(\hat{\gamma_j}_{u})$ and $\gamma_i$ is $2^r$-uniform, so that $|U\gamma_i|\ge 2^{m-r}+1$. In particular $\dim(U\gamma_i)\ge m-r+1$, in contradiction to the assumption that $\gamma_i$ is strongly $(r-1)$-anti-invariant.
\end{proof}

The next corollary is an immediate consequence of Corollary \ref{cormain2}, Remark \ref{present} and Remark \ref{recprin}.

\begin{corollary}
The round functions of PRESENT, RECTANGLE and PRINTcipher generate the alternating group.
\end{corollary}

\section{Final remarks}\label{sec:finrem}

Vectorial Boolean functions used as S-boxes in block ciphers must have low uniformity to prevent linear cryptanalysis (see \cite{lincrit}), differential cryptanalysis (see \cite{crittoanalisi1,crittoanalisi2}) and high non-linearity. Differentially $2$-uniform functions, also called \emph{Almost Perfect Non-linear (APN) functions}, are optimal. The most common dimensions used for S-boxes are even, often powers of $2$. Unfortunately, no 4-bit APN permutation exists and no APN permutation on $(\FF_2)^m$ with $m$ even was known until in \cite{apnseidim} it has been constructed a $6$-bit APN permutation (still now this is the only example of APN permutation in even dimension). For this reason, the permutations used as S-boxes in block ciphers are $4$-uniform.

The following computational result has been obtained looking at all the affine equivalence classes of $4$-bit S-Boxes.

\begin{fact}\label{4uni-1anti}
Let $\gamma:(\mathbb F_2)^4 \rightarrow (\mathbb F_2)^4$ be a $4$-bit S-Box. If $\gamma$ is $4$-uniform, then it is strongly $1$-anti-invariant.
\end{fact}

It would be interesting to give a direct proof of Fact \ref{4uni-1anti}. Indeed, by Theorem \ref{primitivity} and Corollary \ref{cormain2}, this implies that the group generated by the round functions of a tb cipher with a strongly proper round and bricks of dimension 4 is the alternating group, provided that the bricks of the strongly proper round are $4$-uniform.

Next we relate the strongly 1-anti-invariance with the non-linearity, that is another algebraic property of vectorial Boolean functions.
More precisely, we will show that, for a permutation, the strongly 1-anti-invariance is equivalent to the fact that the permutation has non-linearity greater than $0$.

For any $m\geq 1$, let $f$ and $g$ be $m$-variable Boolean functions. The Hamming distance between $f$ and $g$ is
$$\mathrm{d}(f,g)=|\{x\in (\FF_2)^m\mid f(x)\neq g(x)\}|.$$
Recall also that $f$ can be represented as a multivariate polynomial over $\FF_2$. This polynomial is called the Algebraic Normal Form (ANF) of $f$ (see for instance \cite{carlet}) and its degree is the (algebraic) degree of $f$ $\deg(f)$.
The affine functions are Boolean functions of degree at most $1$, and they form a set which we denote by $\cA_m$.
The {\it non-linearity} of $f$ is then given by
$$\mathcal{N}(f)={\rm min}\{{\rm d}(f,\alpha)\mid \alpha \in \cA_m\}.$$
For a vectorial Boolean function $f:(\mathbb{F}_2)^m \to (\mathbb{F}_2)^m$, we denote by $_vf$ the component $\sum_{i=1}^m v_if_i$ of $f$, where $f_1,\dots,f_m$ are the coordinate functions of $f$, for all $v\in(\mathbb{F}_2)^m\backslash\{0\}$. Thus, we have
$$
\mathcal{N}(f)={\rm min}\{\mathcal N(_vf)\mid v\in(\mathbb{F}_2)^m\backslash\{0\}\}.
$$

The following result is Proposition 5 of \cite{Ca16}.

\begin{proposition}\label{prop:Va}
Let $f:(\mathbb{F}_2)^m\rightarrow(\mathbb{F}_2)^m$ be a vectorial Boolean function and $a\in (\mathbb{F}_2)^m\backslash\{0\}$. Let $V_a$ be the vector space $\{v\,\in(\mathbb{F}_2)^m\,|\,\deg(_v\hat{f}_a)=0\}$. Then $af+V_a^{\perp}$ is the smallest affine subspace of $(\mathbb{F}_2)^m$ containing $\Im(\hat{f}_a)$.
\end{proposition}

We are now ready to prove the above announced result.

\begin{proposition}\label{nonlinanti}
Let $m>1$ and let $\gamma$ be a permutation of $(\mathbb{F}_2)^m$ such that $0\gamma=0$. Then $\mathcal N(\gamma)\ne 0$ if and only if $\gamma$ is strongly 1-anti-invariant.
\end{proposition}

\begin{proof}
Let $\mathcal N(\gamma)\ne 0$ and suppose that $\gamma$ is not strongly 1-anti-invariant. Then there exist $U,\,W<(\mathbb{F}_2)^m$ such that $\dim(U)=\dim(W)=m-1$ and $U\gamma=W$. As the non-linearity of a vectorial Boolean function is affine invariant (see for instance \cite{carlet}), without loss of generality, we may assume $U=W=\langle e_1,\ldots,e_{m-1}\rangle$, where $e_i$ denotes the vector with $1$ in the $i$-th position and $0$ elsewhere. Thus $(\mathbb{F}_2)^m=U\cupdot (U+e_m)$, in other words $U$ and $(U+e_m)$ give a partition of $(\mathbb{F}_2)^m$.
Let $u \in U$ with $u\neq 0$. Since $U\gamma=U$, we have $(U+e_m)\gamma=U+e_m$ which implies that $(u+v)\gamma+v\gamma \in U$, for all $v\in (\mathbb{F}_2)^m$. Hence $\Im(\hat{\gamma}_u)\subset U$ and, by Proposition \ref{prop:Va}, we have $V_u^\perp+u\gamma\subseteq U$. But $u\gamma\in U$, so that $V_u^\perp\subseteq U$ and $U^\perp=\langle e_m\rangle\subseteq V_u$. It follows that $\langle e_m\rangle\subseteq V_{e_i}$ for all $e_i$ with $1\le i\le m-1$, i.e., the derivate $(\widehat{{\gamma}_{m}})_{e_i}$ is constant. This implies that in the ANF representation of $\gamma_{m}$ the variable $x_i$ could appear only in the part of degree $1$, that is $\gamma_{m}$ has at most degree $1$, in contrast to the fact that $\cN(\gamma)\ne 0$.

Conversely, let $\gamma$ be  strongly 1-anti-invariant and suppose $\cN(\gamma)=0$. As before, let $U=\langle e_1,\ldots,e_{m-1}\rangle$ and $(\mathbb{F}_2)^m=U\cupdot (U+e_m)$. Without loss of generality, we may assume $\gamma_{m}=x_m$. Since $\gamma$ is a permutation, it is easy to verify that $U\gamma=U$, which is a contradiction.
\end{proof}

By the previous proposition an S-Box $\gamma$ is strongly 1-anti-invariant if $\gamma$ does not have any affine component.

In general, the minimum request to prevent linear and differential attacks on a block cipher is the 4-differential uniformity and the nonzero non-linearity of S-Boxes. The lightweight translation based ciphers are often designed using $m$-bit S-Boxes with $m<5$, in order to have a low implementation complexity in hardware. So, for this type of ciphers, by Corollary \ref{cormain2} and Proposition \ref{nonlinanti}, the choice of a strongly proper mixing layer allows to avoid some attacks based on the order or the imprimitivity of the group generated by the round functions.

\subsection*{Acknowledgements} The authors would like to thank Ralph Wernsdorf for his useful comments after reading an earlier version of the manuscript.

\end{document}